\documentclass[reqno,12pt]{article}

\usepackage{a4wide}
\usepackage{amsmath} 
\usepackage{amssymb}
\usepackage{amsthm}
\usepackage[utf8]{inputenc} 
\usepackage{graphicx} 
\usepackage[english, russian]{babel}
\usepackage{xcolor}

\numberwithin{equation}{section}

\newtheorem{theo}{Theorem}
\newtheorem*{theo*}{Theorem}
\newtheorem{conj}{Conjecture}

\newtheorem{lem}{Lemma}
\newtheorem{defi}{Definition}
\newtheorem*{defi*}{Definition}

\theoremstyle{remark}

\newcommand{\Qbar}{\overline{\mathbb Q}}

\begin{document}

 \selectlanguage{english}

\title{On primary pseudo-polynomials (Around Ruzsa's Conjecture)}
\date\today
\author{\'E. Delaygue and T. Rivoal}
\maketitle

\begin{abstract} Every polynomial $P(X)\in \mathbb Z[X]$ satisfies the congruences $P(n+m)\equiv P(n) \mod m$ for all  integers $n, m\ge 0$. An integer valued sequence $(a_n)_{n\ge 0}$ is called a pseudo-polynomial when 
it satisfies these congruences. Hall characterized pseudo-polynomials and 
proved that they are not necessarily polynomials.  
A long standing conjecture of Ruzsa says that a pseudo-polynomial $a_n$ is a polynomial as soon as $\limsup_n \vert a_n\vert^{1/n}<e$. Under this growth assumption, Perelli and Zannier proved that the generating series $\sum_{n=0}^\infty a_n z^n$ is a $G$-function. A primary pseudo-polynomial  is an integer valued sequence $(a_n)_{n\ge 0}$ such that $a_{n+p}\equiv a_n \mod p$ for all integers $n\ge 0$ and all prime numbers $p$. The same conjecture has been formulated for them, which implies Ruzsa's, and this paper revolves around this conjecture. We obtain a Hall type characterization of primary pseudo-polynomials and draw various consequences from it. We give a new proof and generalize a result due to Zannier that any primary pseudo-polynomial with an algebraic generating series is a polynomial. This leads us to formulate a conjecture on diagonals of rational fractions and primary pseudo-polynomials, which is related to classic conjectures of Christol and van der Poorten. We make the Perelli-Zannier Theorem 
effective. We prove a P\'olya type result: if there exists a function $F$ analytic in a right-half plane  with not too large exponential growth (in a precise sense) and such that for all large $n$ 
the primary pseudo-polynomial $a_n=F(n)$, then $a_n$ is a polynomial.
Finally, 
we show how to construct a non-polynomial primary pseudo-polynomial starting from any primary pseudo-polynomial generated by a $G$-function different of $1/(1-x)$.
\end{abstract}

\section{Introduction}

Let $P(X)$ be a polynomial in $\mathbb{Z}[X]$. For all distinct integers $m$ and $n$, the integer $m-n$ divides $P(m)-P(n)$. Equivalently, for all integers $n$ and $k$ with $k\neq 0$, we have the congruence $P(n+k)\equiv P(n)\mod k$. 

In the sense of Hall \cite{hall}, a sequence $(a_n)_{n\ge 0}\in \mathbb Z^{\mathbb N}$  is said to be a {\em pseudo-polynomial} when the following property holds: for any integers $n\geq 0$ and $k\geq 1$, we have $a_{n+k}\equiv a_n \mod k$.

Note that an integer-valued polynomial is not necessarily a pseudo-polynomial as $X(X+1)/2$ shows, see below. Pseudo-polynomials have long been studied for themselves, but they have also found recent applications in analytic number theory \cite{kow, kup}.
\medskip

For every non-negative integer $k$, we consider the polynomial
$$
P_k(X)=\frac{X(X-1)\cdots(X-k+1)}{k!},
$$
and $P_0(X):=1$, whose integer values are binomial coefficients. It is well known \cite[Problem 85]{polya} that integer valued polynomials are $\mathbb Z$-linear combinations of the $P_k$'s. It turns out that those polynomials also lead to a characterization of pseudo-polynomials. Hall proved in \cite{hall} that a sequence $(a_n)_{n\geq 0}$ is a pseudo-polynomial if and only if there exists a sequence of integers $(b_n)_{n\geq 0}$ such that, for every positive integer $n$, $b_n$ is a multiple of $d_n:=\mathrm{lcm}\{1,2,\dots,n\}$ (with $d_0:=1$) and we have
\begin{equation}\label{eq:ajout1507}
a_n=\sum_{k=0}^\infty b_kP_k(n).
\end{equation}
(For each given $n\ge 0$, the sum is finite, more precisely it runs from $k=0$ to $k=n$.)
Given a sequence $(a_n)_{n\ge 0}\in \mathbb C^{\mathbb N}$, we define its 
\textit{binomial transform} $(b_n)_{n\ge 0}\in \mathbb C^{\mathbb N}$ as  

\begin{equation}\label{eq:binomialtransform}
\forall n\ge 0, \quad b_n:=\sum_{k=0}^n (-1)^{n-k} \binom{n}{k}a_k.
\end{equation}
It is well-known that $(a_n)_{n\ge 0}$ can be recovered from $(b_n)_{n\ge 
0}$ by 
\begin{equation}\label{eq:involution}
\forall n \ge 0, \quad a_n=\sum_{k=0}^n \binom{n}{k}b_k, 
\end{equation}
\textit{i.e.} the binomial transform is ``almost'' involutive.  Hence Hall proved that a sequence $(a_n)_{n\geq 0}$ is a pseudo-polynomial if and only if its binomial transform $(b_n)_{n\geq 0}$ satisfies
\begin{equation} \label{eq:hall} \forall n\ge 0, \; d_n\,\vert\, b_n.
\end{equation}
Because of \eqref{eq:binomialtransform} and \eqref{eq:involution}, observe that in \eqref{eq:ajout1507} the sequence $(b_n)_{n\ge 0}$ is uniquely determined by the sequence $(a_n)_{n\ge 0}$, so that as claimed above  $\frac{1}{2}X(X+1)=P_1(X)+P_2(X)$ is not a pseudo-polynomial because $d_2=2$.

An important property of the binomial transform is that they also lead to a characterization of polynomials. Let $(a_n)_{n\geq 0}$ be a complex sequence and $(b_n)_{n\geq 0}$ its binomial transform. Then the following assertions are equivalent:
\begin{itemize}
\item there exists $P(X)\in \mathbb C[X]$ such that $a_n=P(n)$ for all integers $n\geq 0$;
\item $b_n=0$ for all $n$ large enough.
\end{itemize}
(We can replace $\mathbb C[X]$ by $\mathbb Q[X]$ when the sequences take values in $\mathbb Q$.)

{\em In the sequel, we shall often say that an integer valued sequence $(a_n)_{n \ge 0}$ is ``a polynomial" when there exists a polynomial 
$P(X)\in \mathbb{Q}[X]$ such that $a_n=P(n)$ for all non-negative integers $n$. }
\medskip

We recall that, for all non-negative integers $n$, we have
$$
d_n=\prod_{p\le n} p^{\lfloor\log_p(n)\rfloor} \le 3^n\quad\textup{and}\quad d_n^{1/n}\underset{n\to +\infty}{\longrightarrow} e,
$$
by the Prime Number Theorem. Taking $b_n:=n!$ in \eqref{eq:hall}, we see from \eqref{eq:involution} that  the resulting pseudo-polynomial  $a_n$ is 
simply equal to $\lfloor n!e\rfloor$ for $n\ge 1$,  
which is obviously not a polynomial. With $b_n:=d_n$ in \eqref{eq:hall}, we obtain from \eqref{eq:involution} another pseudo-polynomial of slower growth $a_n:=\sum_{k=0}^n \binom{n}{k}d_k \le 4^n$; since $a_n\ge 2^n$ for all $n\ge 0$, this is not a polynomial either.

Following Hall, we say that a non-polynomial pseudo-polynomial is a \textit{genuine} pseudo-polynomial. The search of minimal growth conditions that can be attained by genuine pseudo-polynomials has been the subject of many papers. Hall \cite{hall} and Ruzsa \cite{ruzsa} independently proved  that if 
\begin{equation}\label{eq:hallruzsa}
\limsup_{n\to +\infty} \vert a_n\vert^{1/n} < e-1,
\end{equation}
then $(a_n)_{n\geq 0}$ is a polynomial. Using his characterization \eqref{eq:hall}, Hall~\cite[p.~76]{hall} sketched an inductive construction of a genuine pseudo-polynomial $(a_n)_{n\ge 0}$ such that
$$
\limsup_{n\to +\infty} \vert a_n\vert^{1/n} \le e. 
$$
Ruzsa proposed the following conjecture.

\begin{conj}[Ruzsa]\label{conj:ruzsa}
Let $(a_n)_{n\ge 0}\in \mathbb Z^{\mathbb N}$ be a pseudo-polynomial such 
that 
\begin{equation}\label{eq:ruzsa}
\limsup_{n\to +\infty} \vert a_n\vert^{1/n} < e.
\end{equation}
Then $(a_n)_{n\geq 0}$ is a polynomial.
\end{conj} 

In fact, many results towards Ruzsa's conjecture have been proven for sequences we shall call {\em primary pseudo-polynomial} (for lack of better terminology). 

\begin{defi}
A sequence $(a_n)_{n\ge 0}\in \mathbb Z^{\mathbb N}$ is said to be a primary pseudo-polynomial when the following property holds: for any integer $n\ge 0$ and any prime number $p$, $a_{n+p}\equiv a_n \mod p$. 
\end{defi}

The set of primary pseudo-polynomials is a ring for the term-wise sum 
and product of sequences in $\mathbb Z^{\mathbb N}$, with the null and unit sequences  defined with all  terms equal to $0$ and all terms equal to 
$1$  respectively. 
A pseudo-polynomial is a primary pseudo-polynomial but the converse is false (see the comments following Theorem \ref{theo:1} below).  
Many authors delt with the following conjecture, the truth of which would 
imply that of Ruzsa. 
\begin{conj}\label{conj:ruzsaext}
Let $(a_n)_{n\ge 0}\in \mathbb Z^{\mathbb N}$ be a primary pseudo-polynomial such that 
\begin{equation}\label{eq:ruzsaext}
\limsup_{n\to +\infty} \vert a_n\vert^{1/n} < e.
\end{equation}
Then $(a_n)_{n\geq 0}$ is a polynomial.
\end{conj} 
As in the case of pseudo-polynomials, $<e$ cannot be replaced by $\le e$, 
and we refer again to the comments following Theorem \ref{theo:1} for a proof of this. 

Perelli and Zannier \cite{pz2} proved a highly non-trivial property: under the growth condition \eqref{eq:ruzsa} in Conjecture \ref{conj:ruzsaext}, the primary pseudo-polynomial sequence $(a_n)_{n\ge 0}$ satisfies a linear recurrence with polynomials coefficients; see \eqref{eq:recurrence2} below. In other words, 
the generating function $f_a(x):=\sum_{n=0}^\infty a_n x^n \in \mathbb Z[[x]]$ satisfies a linear differential equation with coefficients in $\mathbb Z[x]$. Hence, $f_a$ is a $G$-function~(\footnote{A power series $\sum_{n=0}^\infty a_n x^n \in \Qbar[[x]]$ is said to be a $G$-function when it is solution of a non-zero linear differential equation over $\Qbar(x)$ ($D$-finiteness), and the maximum of the modulus of all the Galoisian conjuguates of $a_0, \ldots, a_n$ as well as the positive denominator of $a_0, \ldots, a_n$ are both bounded for all $n\ge 0$ by $C^{n+1}$, for 
some $C\ge 1$.  For instance, any $D$-finite series in $\mathbb Z[[x]]$ with  positive radius of convergence is a $G$-function. A power series $\sum_{n=0}^\infty \frac{a_n}{n!} z^x \in \Qbar[[x]]$ is said to be an $E$-function when $\sum_{n=0}^\infty a_n x^n$ is a $G$-function.   See \cite{andre, shidlovskii} for the  properties satisfied by these functions.}). Perelli and Zannier also proved a form of Conjecture \ref{conj:ruzsaext} under a stronger assumption than \eqref{eq:ruzsa}, \textit{i.e.} with $e$ replaced by $e^{0.66}$. Zannier \cite{zannier} was even able to replace $e^{0.66}$ by $e^{0.75}$.
\medskip

Zannier \cite[p. 398]{zannier} also proved that we can omit the growth condition if we further assume that $f_a(x)$ is algebraic. That is, if $(a_n)_{n\geq 0}$ is a primary pseudo-polynomial such that $f_a(x)$ is algebraic then $(a_n)_{n\geq 0}$ is a polynomial. Diagonals of rational fractions form an intermediate class between algebraic series and $G$-functions. By definition, the diagonal of a multivariate power series $\sum_{n_1, \ldots, n_k\ge 0}u_{n_1, \ldots, n_k} z_1^{n_1}\cdots z_k ^{n_k}$ is defined by $\sum_{n = 0}^\infty u_{n, n, \ldots, n} z^{n}$. 
A classical result of Furstenberg~\cite{furstenberg} says that algebraic series over a field coincide with diagonals of rational functions in two variables. It would be very interesting to know whether the following conjecture holds true.

\begin{conj}\label{conj diag}
Let $(a_n)_{n\geq 0}\in\mathbb{Z}^\mathbb{N}$ be a primary pseudo-polynomial such that its generating series is the diagonal of a rational fraction. Then $(a_n)_{n\geq 0}$ is a polynomial.
\end{conj}

Moreover, diagonals of rational functions are globally bounded $G$-functions in the sense of Christol, who   conjectured that the converse holds (see \cite{christol}). In particular, $G$-functions with integer coefficients are globally bounded. Therefore given the theorem of Perelli and Zannier (recalled below), Christol's Conjecture and Conjecture \ref{conj diag} would 
together imply Conjecture~\ref{conj:ruzsaext} and Ruzsa's Conjecture \ref{conj:ruzsa}. See also related comments in \cite[pp.~392--393]{zannier}.

Furthermore, Conjecture 3 is implied by the following special case of van der Poorten's Conjecture \cite[p. 13]{VDPconj} : Given $f(x)$ the diagonal of a rational fraction, if for almost all primes $p$ the reduction of $f(x)$ modulo $p$ is a rational fraction, then $f(x)$ is a rational fraction.

Hence Conjecture \ref{conj diag} could be seen as an intermediate step towards the proof of Rusza's conjecture. Note that Christol's Conjecture together with van der Poorten's Conjecture imply Rusza's Conjecture.
\medskip

In this paper, we are interested in the properties of  primary pseudo-polynomials and of their generating functions. We now present our four main results, make comments about their significance and mention further open problems.

\medskip

\noindent {\bf {\em $\bullet$ A Hall type characterization of primary pseudo-polynomials.}}
 We shall first prove an analogue (\textit{i.e.} Eq. \eqref{eq:1} below)
of Hall's characterization for pseudo-polynomials  and deduce some consequences of it. We set $P_0=P_1:=1$ and, for $n\geq 2$, 
$$
P_n:=\prod_{p\le n} p,
$$ 
where the product is over prime numbers. By the Prime Number Theorem, we have $P_n^{1/n} \to e$ as $n$ tends to $+\infty$. We also say that a primary pseudo-polynomial which is not a polynomial is a \textit{genuine} primary pseudo-polynomial.

\begin{theo}\label{theo:1} We have the following.
\begin{itemize}
\item[$(i)$] A sequence 
$(a_n)_{n\ge 0}\in \mathbb Z^{\mathbb N}$ is a primary pseudo-polynomial if and only if its binomial transform $(b_n)_{n\ge 0}\in \mathbb Z^{\mathbb N}$ satisfies
\begin{equation} \label{eq:1}
\forall n\ge 0, \; P_n\,\vert\, b_n.
\end{equation}

\item[$(ii)$] Given a genuine primary pseudo-polynomial $(a_n)_{n\ge 0}$, then $$\liminf_{n\to+\infty } \vert b_n\vert^{1/n} \ge e.$$

\item[$(iii)$] If a primary pseudo-polynomial $(a_n)_{n\ge 0}$ satisfies $\limsup_{n\to +\infty} \vert a_n\vert ^{1/n}<e-1$, then $(a_n)_{n\geq 0}$ is a polynomial.

\item[$(iv)$] Given any function $\varphi:\mathbb N\to \mathbb R$ with $\varphi(0)=1$, there exists a genuine primary pseudo-polynomial $(A_n)_{n\ge 0}$ such that $\varphi(n)\le A_n\le \varphi(n)+2P_n$ for all $n\in \mathbb N$.
\end{itemize}
\end{theo}

We recall that 
$$
d_n=\prod_{p\le n} p^{\lfloor\log_p(n)\rfloor}. 
$$
Hence, 
for all $n\ge 0$, $P_n$ divides $d_n$, but obviously $d_n$  divides $P_n$ 
for no $n\ge 4$. Choosing $b_n:=P_n$ in \eqref{eq:1}, the resulting sequence in \eqref{eq:involution}  $a_n:=\sum_{k=0}^n \binom{n}{k}P_k$ is a primary pseudo-polynomial, but not a pseudo-polynomial because it does not satisfy Hall's criterion \eqref{eq:hall}. Under the assumption in $(ii)$,  if we also assume that $b_n$ is eventually of the same sign, then 
$
\liminf_{n\to +\infty}\vert a_n\vert^{1/n} \ge e+1
$
because $a_n=\sum_{k=0}^n \binom{n}{k}b_k$. 
Consequently, any putative counter-example $(a_n)_{n\ge 0}$ to Conjecture~\ref{conj:ruzsaext}  must be such that its binomial transform $(b_n)_{n\ge 0}$ changes sign infinitely often. A similar remark applies to Ruzsa's 
Conjecture~\ref{conj:ruzsa}.

Assertion $(iii)$ is the analogue of the Hall-Ruzsa result recalled at Eq. \eqref{eq:hallruzsa}.

In $(iv)$, given $\varphi$, the existence of sequence $(A_n)_{n\ge 0}$ is 
proved constructively by an inductive process. An important consequence of $(iv)$ is the existence of  a genuine primary pseudo-polynomial of any growth $\varphi(n)>0$ provided $\liminf_n \varphi(n)^{1/n} \ge e$. 
In particular,  with $\varphi(n)=P_n$, we deduce that $<e$ cannot be replaced by $\le e$ on the right-hand side of  \eqref{eq:ruzsaext} in Conjecture~\ref{conj:ruzsaext}.

\medskip

\noindent {\bf {\em $\bullet$ Primary pseudo-polynomials with an algebraic generating series.}}
The generating functions $f_a$ and $f_b$ of the sequences $(a_n)_{n\ge 0}$ and its binomial transform $(b_n)_{n\ge 0}$ satisfy the relations
\begin{equation}\label{eq:fabalg}
f_b(x)=\frac{1}{1+x}f_a\Big(\frac{x}{1+x}\Big)\quad\textup{and}\quad f_a(x)=\frac{1}{1-x}f_b\Big(\frac{x}{1-x}\Big).
\end{equation}
In particular $f_a(x)$ is algebraic over $\mathbb Q(x)$ if and only if $f_b(x)$ is algebraic over $\mathbb Q(x)$.

\begin{theo}\label{theo:2} 
Let $(a_n)_{n\ge 0}\in \mathbb Z^{\mathbb N}$ be a primary pseudo-polynomial, and let $(b_n)_{n\ge 0}\in \mathbb Z^{\mathbb N}$ be its binomial transform.

\begin{enumerate} 
\item[$(i)$] Assume there exists $m\ge 0$ such that  
$f_a^{(m)}(x)$ 
is algebraic over $\mathbb Q(x)$. Then   
$(a_n)_{n\geq 0}$ is a polynomial, and thus $f_a(x)\in \mathbb Q(x)$.

\item[$(ii)$] If 
$f_b(x)$  
is algebraic over $\mathbb Q(x)$, then $f_b(x)$ is in $\mathbb Z[x]$.
\end{enumerate}
\end{theo} 
As said above,  the case $m=0$ in $(i)$ was proved by Zannier in \cite[p. 398]{zannier}. We shall present a different proof of this case. It implies that if there exists a counter example to Conjecture 
\ref{conj:ruzsaext} or to Ruzsa's Conjecture \ref{conj:ruzsa}, 
then its generating function $f_a(x)$ is transcendental over $\mathbb C(x)$. 

\medskip

\noindent {\bf {\em $\bullet$ 
An effective version of a result of Perelli and Zannier.}}
 In \cite {pz2}, Perelli and Zannier sketched the proof of the following result.
\begin{theo*}[Perelli-Zannier]
Let $(a_n)_{n\ge 0}$ be a primary pseudo-polynomial such that there exist 
$c>0$ and $1< \delta<e$ such that $\vert a_n\vert \le c \delta^n$ for all 
$n\ge 0$. Then there exist an integer $S\ge 0$ and $S+1$ polynomials $p_0(X), \ldots, p_S(X)\in \mathbb Z[X]$ not all zero such that, for all $n\geq 0$, we have
\begin{equation}\label{eq:recurrence2}
\sum_{j=0}^S p_j(n) a_{n+j}=0.
\end{equation}
In other words, $f_a(x)$ is $D$-finite, and even a $G$-function.
\end{theo*}  
Perelli and Zannier mentioned that it would be possible to provide  upper 
bounds for $S$ and the degree/height of the $p_j(X)$ in terms of $\delta$, but they did not write them down. We make more precise their theorem as 
follows, where given $Q(X)=\sum_{j} q_j X^j \in \mathbb C[X]$, we set $H(Q):=\max_j \vert q_j\vert$. 
\begin{theo}\label{theo:recurrence} Let $(a_n)_{n\ge 0}$ be a primary pseudo-polynomial such that there exist 
$c>0$ and $1< \delta<e$ such that $\vert a_n\vert \le c \delta^n$ for all $n\ge 0$. Then, there exists an effectively computable constant $H(c,\delta)\ge 1$ such that a non-trivial linear recurrence for $(a_n)_{n\ge 0}$ as in  \eqref{eq:recurrence2} holds with 
\begin{equation}\label{eq:upperbounds}
\begin{cases}
\max_j\deg(p_j) \le  \max\left(0,\left\lceil\frac{5\log(\delta)-1}{1-\log(\delta)}\right\rceil\right),
\\
\max_j H(p_j) \le H(c,\delta),
\\
S\le  \log\big(H(c,\delta)\big)/\log(\delta).
\end{cases}
\end{equation}
Moreover, the Perelli-Zannier Theorem is best possible in the sense that its conclusion  does not necessarily hold if  $e$ is replaced by any larger number in the assumption $1< \delta<e$.
\end{theo}

To prove the final statement in Theorem \ref{theo:recurrence}, we take $\varphi(n):=P_n$ in Theorem~\ref{theo:1}$(iv)$: we obtain a genuine primary pseudo-polynomial $A_n$ such that $\vert A_n\vert^{1/n} \to  e$. Hence $(A_n)_{n\ge 0}$ does not satisfy a non-zero linear recurrence with coefficients polynomials over $\mathbb{Q}$.~(\footnote{Indeed, if a solution $(a_n)_{n\ge 0}$ of a 
linear  recurrence with coefficients polynomials over $\mathbb{Q}$ is such that $\vert a_n\vert^{1/n}\to \alpha$ finite, then $\alpha$ is an algebraic number. 

Moreover, if a sequence of {\em rational numbers} satisfies a non-zero linear recurrence of minimal order with coefficients polynomials over $\mathbb{C}$, then these coefficients are necessarily polynomials over $\mathbb Q$, up to a common non-zero multiplicative constant. Hence $(d_n)_{n\ge 0}$ 
and  $(P_n)_{n\ge 0}$ do not satisfy any non-zero linear recurrence with coefficients polynomials over $\mathbb{C}$. 
}) 

Lower and upper bounds for the function $H(c,\delta)$ are given in    
\eqref{eq:boundlowH} and \eqref{eq:boundupH} respectively 
in \S\ref{ssec:conclusion1}. Our bound for $\max\deg(p_j)$ in \eqref{eq:upperbounds} is obviously not optimal; in fact, at the cost of more complicated computations, Perelli-Zannier \cite{pz2} and then Zannier \cite{zannier} obtained better bounds when $\delta\le e^{0.66}$ and $\delta\le e^{0.75}$. 

The classification of primary pseudo-polynomials with a $D$-finite generating series is an open problem. 
As shown by the above example $(A_n)_{n\ge 0}$, 
Theorem \ref{theo:1} rules out the possibility that {\em every} primary pseudo-polynomial satisfies a linear recurrence with coefficients polynomials over $\mathbb{Q}$. 
Another example is  
 the primary pseudo-polynomial 
$D_n:=\sum_{k=0}^n \binom{n}{k} P_k$: it cannot satisfy such a linear 
recurrence because otherwise
$$
P_n=\sum_{k=0}^n (-1)^{n-k}\binom{n}{k} D_k 
$$
would satisfy one as well, which is not possible because $P_n^{1/n}\to e$. 
Since $P_n\ge 0$ and $P_n=e^{n+o(n)}$, a simple 
analytic argument shows that $D_n=(e+1)^{n+o(n)}$. 
More specifically, it would be interesting to know if there is any genuine primary pseudo-polynomial $(a_n)_{n\ge 0}$ such that $f_a(x)$ is 
a $G$-function. There exist primary pseudo-polynomials $(a_n)_{n\ge 0}$ such that $f_a$ is $D$-finite but is not a $G$-function. For instance, 
$$
e_n:=\lfloor (n+1)!e\rfloor=\sum_{k=0}^{n+1} \binom{n+1}{k}k!\quad (n\ge 
0)
$$ 
is a primary pseudo-polynomial by Theorem \ref{theo:1}, and for all $n\ge 0$, 
$$
e_{n+2}=(n+4)e_{n+1}-(n+2)e_{n}\quad (e_0=2, e_1=5),
$$
so that $\sum_{n= 0}^\infty e_n x^n$ is $D$-finite.  A method to obtain further examples is presented at the end of 
\S\ref{sec:prooftheo51}.

\medskip

\noindent {\bf {\em $\bullet$ A P\'olya type result for primary pseudo-polynomials.}}
Perelli and Zannier  also proved in \cite{pz} that if a (primary) pseudo-polynomial $a_n=F(n)$ for some entire function $F$ such that 
$$
\limsup_{R\to \infty} \frac{1}R\log \max_{\vert x\vert= R}\vert F(x)\vert<\log(e+1),
$$
then $(a_n)_{n\geq 0}$ is a polynomial. 
We prove here a result of a similar flavor with a different analyticity condition. 
\begin{theo}\label{theo:3} Let  
$(a_n)_{n\ge 0}\in \mathbb Z^{\mathbb N}$ be a primary pseudo-polynomial. 
Let us assume that there exists $F(x)$ analytic in a right-half plane $\Re(x)>u$  such that $a_n=F(n)$ for all $n>u$, and $c>0, 0<\rho< \log(2\sqrt{e})$ such that
\begin{equation}\label{eq:2}
\big\vert F(x)\big\vert \le  c\cdot e^{\rho \Re(x)}
\end{equation}
for $\Re(x)> u$. Then $F(x)$ is  in $\mathbb Q[x]$ and $(a_n)_{n\geq 0}$ is a polynomial.
\end{theo} 
We have $\log(2\sqrt{e})\approx 1.193$ while $\log(e+1)\approx 1.313$. In 
the proof, we shall obtain that $(a_n)_{n\geq 0}$ is a polynomial before proving that $F$ is a polynomial. Because of the Perelli-Zannier Theorem recalled before Theorem \ref{theo:recurrence}, the assumptions of  Theorem~\ref{theo:3} are in fact natural in the context of Ruzsa's Conjecture~\ref{conj:ruzsa}. Indeed, for any given $G$-function $f(x):=\sum_{n=0}^\infty v_n x^n \in \Qbar[[x]]$, there exists a function 
$$
\lambda(x) :=  
\sum_{j=1}^p c_j(x)\cdot  e^{\rho_j x},
$$
for some functions $c_j(x)$ analytic in $\Re(x)>u$ and of polynomial growth (at most), and such that $v_n=\lambda(n)$ for all $n>u$; the numbers $e^{-\rho_j }$ are the finite singularities of $f(x)$ (see \cite[\S 7.1]{firi} for details). Notice that a bound involving $e^{\rho_j \Re(x)}$ is {\em a priori} different of a bound involving $\vert e^{\rho_j x}\vert=e^{\Re(\rho_j x)}$, but they are the same when $\rho_j\in \mathbb R$. In particular, if all the singularities of $f(x)$ are positive real numbers, then a bound as in \eqref{eq:2} holds for $\lambda(x)$ for some $\rho\in \mathbb R$.

In Theorem \ref{theo:1}$(iv)$, take $\varphi(n)=\delta^{n}$ with $e<\delta<2\sqrt{e}$. This yields a genuine primary pseudo-polynomial $(a_n)_{n\ge 0}$ such that $a_n =\delta^{n+o(n)}$ as $n\to +\infty$. Theorem \ref{theo:3} thus implies that there is no function $F(x)$ analytic in a right-half plane on which \eqref{eq:2} holds, and such that $a_n=F(n)$ for all large $n$.

Perelli-Zannier's result and Theorem \ref{theo:3} are similar to P\'olya's celebrated theorem: if an entire function 
$F(x)$ is such that 
$$
\limsup_{R\to +\infty} \frac{1}R
\log\max_{\vert x\vert= R}\vert F(x)\vert<\log(2)\quad\textup{and}\quad F(\mathbb N)\subset 
\mathbb Z,
$$ 
then $F(x)$ is a polynomial. See \cite{waldschmidt} for a recent survey on P\'olya type results, where a connection with Rusza's Conjecture~\ref{conj:ruzsa} is also  mentioned.

\medskip

Theorems \ref{theo:1}, \ref{theo:2}, \ref{theo:recurrence} and \ref{theo:3} are proved in \S\ref{sec:prooftheo1}, \S\ref{sec:prooftheo2}, \S\ref{sec:prooftheorec} and \S\ref{sec:prooftheo3} respectively. In \S\ref{sec:prooftheo51}, we present a method to construct a genuine primary pseudo-polynomial starting from every primary pseudo-polynomial generated by a $G$-function (Theorem \ref{theo:51}).

\section{Proof of Theorem \ref{theo:1}}\label{sec:prooftheo1}

$(i)$ Let $(a_n)_{n\geq 0}$ be a sequence of integers and consider its binomial transform $(b_n)_{n\geq 0}$.
\medskip

Assume that for every non-negative integer $n$, $P_n$ divides $b_n$. Let $p$ be a fixed prime number. Hence, for every integer $n\geq p$, $p$ divides $b_n$. For every non-negative integer $n$, it yields
\begin{align*}
a_{n+p}&=\sum_{k=0}^{n+p}\binom{n+p}{k}b_k
\\&\equiv \sum_{k=0}^{p-1}\binom{n+p}{k}b_k\mod p
\\
&\equiv \sum_{k=0}^{p-1}\binom{n}{k}b_k\mod p
\\ 
&\equiv a_n\mod p,
\end{align*} 
where we used Lucas' congruence for binomial coefficients: for every $u,v$ in $\{0,\dots,p-1\}$ and every non-negative integers $m$ and $\ell$, we 
have
$$
\binom{u+mp}{v+\ell p}\equiv\binom{u}{v}\binom{m}{\ell}\mod p.
$$
It follows that $(a_n)_{n\geq 0}$ is a primary pseudo-polynomial.
\medskip

Conversely, assume that $(a_n)_{n\geq 0}$ is a primary pseudo-polynomial. 
Let $p$ be a prime number and $n\geq p$ be an integer. It suffices to show that $p$ divides $b_n$. 

Write $n=v+mp$ with $v$ in $\{0,\dots,p-1\}$ and $m\geq 1$. We obtain that
\begin{align*}
&b_{v+mp}
\\
&=\sum_{k=0}^{v+mp}(-1)^{v+mp-k}\binom{v+mp}{k}a_{k}\\
&=\sum_{u=0}^{p-1}\sum_{\ell=0}^{m-1}(-1)^{v-u}(-1)^{(m-\ell)p}\binom{v+mp}{u+\ell p}a_{u+\ell p}+\sum_{u=0}^v(-1)^{v-u}\binom{v+mp}{u+mp}a_{u+mp}\\
&\equiv\sum_{u=0}^{p-1}\sum_{\ell=0}^{m-1}(-1)^{v-u}(-1)^{(m-\ell)p}\binom{v}{u}\binom{m}{\ell}a_u+\sum_{u=0}^v(-1)^{v-u}\binom{v}{u}a_u\mod 
p\\
&\equiv \sum_{u=0}^v(-1)^{v-u}\binom{v}{u}a_u\sum_{\ell=0}^m(-1)^{(m-\ell)p}\binom{m}{\ell}\mod p \quad \; (\binom{v}{u}=0 \; \textup{for} \; u=v+1, \ldots, p-1)\\
&\equiv \sum_{u=0}^v(-1)^{v-u}\binom{v}{u}a_u(1+(-1)^p)^m\mod p\\
&\equiv 0\mod p\qquad (m\ge 1).
\end{align*}
Hence $p$ divides $b_n$ as expected. It follows that, for every non-negative integer $n$, $P_n$ divides $b_n$. The first equivalence in Theorem \ref{theo:1} is proved.
\medskip

$(ii)$ The binomial transform $(b_n)_{n\geq 0}$ is eventually $0$ if, and 
only if there exists a polynomial $Q(X)$ in $\mathbb{Q}[X]$ such that $a_n=Q(n)$ for every non-negative integer $n$. Hence, if the latter is false, then $(b_n)_{n\geq 0}$ is not eventually $0$ and, since $P_n$ divides $b_n$, it follows that
$$
\liminf_{n\to +\infty}|b_n|^{1/n}\geq e
$$
because $P_n^{1/n} \to e$ as $n\to +\infty.$

\medskip

$(iii)$ If the primary pseudo-polynomial $(a_n)_{n\geq 0}$ is  not a polynomial, then by $(ii)$ above, $\liminf_{n\to +\infty} \vert b_n\vert^{1/n}\ge e$. But since 
$$
b_n=\sum_{k=0}^n (-1)^{n-k}\binom{n}{k} a_k,
$$ 
the assumption $\limsup_{n\to +\infty} \vert a_n\vert^{1/n}<e-1$ implies that 

$$
\limsup_{n\to +\infty} \vert b_n\vert^{1/n}\le \limsup_{n\to +\infty}\bigg(\sum_{k=0}^n \binom{n}{k} \vert a_k\vert\bigg)^{1/n}<e.
$$
This proves that $(a_n)_{n\geq 0}$ is a polynomial.

\medskip

$(iv)$ The following argument generalizes Hall's (sketchy) construction of a genuine pseudo-polynomial with growth $\le e^{n+o(n)}$ in \cite[p.~76]{hall}. By $(i)$, we know that any sequence of integers $(B_n)_{n\ge 0}$ such that $P_n$ divides $B_n$ defines a primary pseudo-polynomial $A_n:=\sum_{k=0}^n \binom{n}{k}B_k$, which is not a polynomial 
if (and only if) $B_n\neq 0$ for infinitely many $n$. Since  $A_ {n-1}$ depends only on $B_0, B_1, \ldots, B_{n-1}$, we will recursively construct 
$B_{n} \neq 0$ and thus $A_{n}$. We have $A_0=B_0$: choosing $B_0=1$, 
we have $\varphi(0)=A_0 =  \varphi(0)+2P_0$. Let $n\ge 0$ and let us assume that we have constructed $B_0, B_1, \ldots, B_{n-1}$ all non-zero and such that $P_{k}$ divides $B_{k}$ and $\varphi(k)\le A_k\le \varphi(k)+2P_{k}$ for all $k\in \{0, \ldots, n-1\}$. 

We want to construct an integer $B_{n}\neq 0$ such that $P_{n}\mid B_{n}$ and $\varphi(n)\le  A_{n}\le \varphi(n)+2P_n$. To do this, we first set 
$C_n:=\sum_{k=0}^{n-1} \binom{n}{k}B_k$, so that we will have $A_{n}=B_{n}+ C_n$. We now perform the euclidean division of $C_n$ by $P_{n}$: we have $C_n=u_nP_{n}+v_n$ with $u_n\in \mathbb Z, v_n\in\mathbb N$ and $0\le v_n<P_{n}$. We set $B_{n}:=w_nP_n \neq 0$ where the non-zero integer $w_n$ is defined as follows: if $\lceil \frac{\varphi(n)-v_n}{P_{n}}\rceil \neq u_n$, we take 
$$
w_n = \left\lceil \frac{\varphi(n)-v_n}{P_{n}}\right\rceil -u_n,
$$
while if $\lceil \frac{\varphi(n)-v_n}{P_{n}}\rceil =u_n$, we take $w_n=1$. Since  $A_{n}=(u_n+w_n)P_{n}+v_n$, we see that 
$\varphi(n)\le A_{n}\le \varphi(n)+P_{n}$ in the former case, while $\varphi(n)+P_{n} \le A_n \le \varphi(n)+2P_{n}$ in the latter case. This finishes the recursive construction of a genuine primary pseudo-polynomial $(A_n)_{n\ge 0}$ such that $\varphi(n)\le A_n\le \varphi(n)+2P_{n}$ for all integers $n\ge 0$.

\section{Proof of Theorem \ref{theo:2}} \label{sec:prooftheo2}

From the proof of Theorem \ref{theo:1}$(i)$, we see that for any given prime number $p$, the assertions ``for all $n\ge 0$, $a_{n+p}\equiv a_n \mod p$'' and ``for all $n\ge p$, $p$ divides $b_n$'' are equivalent. It follows that the assertions ``for all $p\in \mathcal{P}$ and all $n\ge 0$, $a_{n+p}\equiv a_n\mod p$'' and ``for all $p\in \mathcal{P}$ and all $n\ge p$, $p$ divides $b_n$'' are equivalent, where $\mathcal{P}$ is a same set of prime numbers, and this  generalizes Theorem \ref{theo:1}$(i)$. We shall in fact prove  Theorem~\ref{theo:2} under the weaker assumption that there exists an infinite set $\mathcal{P}$ of prime numbers such that for all $p\in \mathcal{P}$ and all $n\ge 0$, $a_{n+p}\equiv a_n\mod p$. 

\medskip

Given $u\in \mathbb Z$ and a prime number $p$, we set $u_{\vert p}:= u \mod p$. Given a power series $F(x):=\sum_{n=0}^\infty u_n x^n$ with integer coefficients, we set 
$$
F_{\vert p}(x):=\sum_{n=0}^\infty u_n{}_{\vert p} x^n \in \mathbb{F}_p[[x]].
$$

\medskip

We shall first prove $(ii)$ for the series $f_b(x):=\sum_{n=0}^\infty 
b_n x^n$.  Let $\mathcal{P}$ denote an infinite set of prime numbers such 
that for all $n\ge 0$ and all $p\in \mathcal{P}$, we have $a_{n+p}\equiv a_n\mod p$. As already said, this is equivalent to the fact that for  all 
$p\in \mathcal{P}$ and all $n\ge p$, $p$ divides $b_n$. It follows in particular that for any $p\in \mathcal{P}$, $f_b{}_{\vert p}(x)$ is a polynomial in $\mathbb{F}_p[x]$ of degree at most $p-1$. For simplicity, we denote by $Q_p(x)$ this polynomial, and by $q_p$ its degree.

Let us now assume that $f_b(x)$ is algebraic over $\mathbb Q(x)$. If $f_b(x)$ is a constant, there is nothing  else to prove. We now assume that $f_b(x)$ is not a constant so that it has degree $d\ge 1$. There exist an integer $\delta\in\{1, \ldots, d-1\}$, some integers $0\le j_1<j_2<\ldots 
<j_{\delta}\le d-1$, and some polynomials $A_d(x)$, $A_{j_1}(x), \ldots, A_{j_{\delta}}(x)\in \mathbb Z[x]$ all not identically zero such that
\begin{equation}\label{eq:alg}
A_df_b^d = \sum_{\ell=1}^\delta A_{j_\ell}f_b^{j_\ell}
\end{equation}
in $\mathbb Z[[x]]$.

We fix $p\in \mathcal{P}$ such that $p>H$ where $H$ is the maximum of the 
modulus of the coefficients of $A_d(x)$, $A_{j_1}(x), \ldots, A_{j_{\delta}}(x)$. It follows that 
\begin{equation}\label{eq:deg}
\deg(A_d{}_{\vert p})=\deg(A_d),\; \deg(A_{j_1}{}_{\vert p})=\deg(A_{j_1}),\; \ldots, \;\deg(A_{j_\delta}{}_{\vert p})=\deg(A_{j_\delta}). 
\end{equation}
We deduce from the reduction of \eqref{eq:alg} mod $p$  that 
\begin{equation}\label{eq:alg2}
A_d{}_{\vert p}Q_p^d = \sum_{\ell=1}^\delta A_{j_\ell}{}_{\vert p}Q_p^{j_\ell}
\end{equation}
in $\mathbb F_p[[x]]$, and in fact in $\mathbb F_p[x]$ because $Q_p(x)\in 
\mathbb F_p[x]$.

\medskip 

\noindent Case 1). If $Q_p$ is identically zero, this means that $p$ divides the coefficients $b_n$ for all $n\ge 0$.

\smallskip
 
\noindent Case 2). If $Q_p$ is not identically zero, we deduce from \eqref{eq:deg} and \eqref{eq:alg2} that 
\begin{align*}
\deg(A_d)+dq_p &\le \max\big(\deg(A_{j_1})+j_1q_p,\deg(A_{j_2})+j_2q_p, \ldots, \deg(A_{j_\delta})+j_\delta q_p\big)
\\
&\le \max\big(\deg(A_{j_1}), \deg(A_{j_2}), \ldots, \deg(A_{j_\delta})\big) + (d-1) q_p.
\end{align*}
Hence
$$
q_p\le \max\big(\deg(A_{j_1}), \deg(A_{j_2}), \ldots, \deg(A_{j_\delta})\big)-\deg(A_d)=:N.
$$
It follows that for any $n>N$, $p$ divides $b_n$, where $N$ is {\em independent} of $p$. 

\noindent Since $p\in \mathcal{P}$ was simply assumed larger than a quantity $H$ depending only on $f_b$, the conclusion of Case 1 and Case 2 is that for any $p\in \mathcal{P}$ such that $ p > H$ and any $n>N$, $p$ divides~$b_n$. Since $\mathcal{P}$ is infinite, $b_n$ is divisible by infinitely many primes when $n>N$. Hence $b_n=0$ for all $n>N$ and $f_b(x)=\sum_{n=0}^N b_nx^n\in \mathbb Z[x]$, as expected.

\medskip

Let us now prove $(i)$ in the case $m=0$. If $f_a(x)$ is algebraic over $\mathbb Q(x)$, then $f_b(x)$ as well by \eqref{eq:fabalg}. Hence $f_b(x)\in \mathbb Z[x]$ by $(ii)$ just proven, \textit{i.e.} there exists an integer $M$ such that 
$b_n=0$ if $n>M$. Since, for all $n\ge 0$, we have
$$
a_n=\sum_{k=0}^{n} \binom{n}{k}b_k=\sum_{k=0}^{\min(n,M)} \binom{n}{k}b_k,  
$$
it follows that for all $n\ge M$, we have $a_n=Q(n)$ with $Q(X)=\sum_{k=0}^{M} \binom{X}{k}b_k \in \mathbb Q[X]$.

\bigskip

We now prove $(i)$ for any integer $m\ge 0$. We need the following simple lemma.
\begin{lem} \label{lem:1412} Let $R(X)\in \mathbb Q(X)$ be such that $R(n)\in \mathbb Z$ for infinitely many integers. Then, $R(X)\in \mathbb Q[X]$.
\end{lem}
\begin{proof} We write $R=A/B$ with $A,B\in \mathbb Q[X]$. We assume that $\deg(B)\ge 1$ otherwise there is nothing to prove. There exist $U,V\in \mathbb Q[X]$ such that $A=UB+V$ and $\deg(V)<\deg(B)$. Let $w\in \mathbb Z\setminus\{0\}$ be such that $wU,wV\in \mathbb Z[X]$. Let $\mathcal{N}$ be the infinite set of integers $n$ such that $R(n)\in \mathbb Z$; without loss of generality, we can assume that $\mathcal{N}$ contains infinitely many positive integers.  For every $n\in \mathcal{N}$, we have 
$wV(n)/B(n)=wR(n)-wU(n) \in \mathbb Z$. But $\lim_{x\to +\infty}wV(x)/B(x)=0$. Hence there exists $M$ such that $n\in \mathcal{N}$ and $n\ge M$ imply that $wV(n)/B(n)=0$. Therefore, $wV$ has infinitely many roots: it must be the null polynomial, so that $R=U\in \mathbb Q[X]$.\end{proof}

Now, we have 
$$
f_a^{(m)}(x)=\sum_{n=0}^\infty (n+m)(n+m-1)\cdots (n+1)a_{n+m} x^{n}.
$$
Since $(a_n)_{n\ge 0}$ is a primary pseudo-polynomial, this is also the case of 
$$
((n+m)(n+m-1)\cdots (n+1)a_{n+m})_{n\ge 0}
$$ 
because it is a product of two primary pseudo-polynomials. Since $f_a^{(m)}(x)$ is algebraic over $\mathbb Q(x)$, $((n+m)(n+m-1)\cdots (n+1)a_{n+m})_{n\geq 0}$ is a polynomial by the already proven case $m=0$ of Theorem~\ref{theo:2}$(i)$. Hence $(a_{n+m})_{n\geq 0}$ is a rational fraction, so that by Lemma \ref{lem:1412}, $(a_{n+m})_{n\geq 0}$ is a polynomial. This completes the proof of Case $(i)$.

\section{Proof of Theorem \ref{theo:recurrence}} \label{sec:prooftheorec}

The proof of the Perelli-Zannier Theorem is based on the following lemma proved in \cite{pz2}. We shall also use it. 
\begin{lem} For $\underline{k}=(k_j)_{j\ge 0} \in (\mathbb{R}^+)^{\mathbb N}$, an integer $N\ge 1$, set 
$$
A(N,\underline{k}):= \{(x_1, \ldots, x_N) \in \mathbb Z^{N} : \vert x_j\vert \le k_j \quad \textup{and}\quad \forall p, \forall n\le N-p, \; x_{n+p}\equiv x_n \mod p\}.
$$
Then 
$$
\#A(N,\underline{k}) \le \prod_{j=1}^N\left(1+\frac{2k_j}{P_{j-1}} \right).
$$
\end{lem} 
Let $R\ge 1, H\ge 0, D\ge 1$ be integers. 
Let $q_0, \ldots, q_R \in \mathbb Z[X]$ with $\max_{j} H(q_j)\le H$ and $\max_j \deg(q_j)\le D-1$. Considering the coefficients of the $q_j$'s as indeterminates, there are $(2H+1)^{RD}$ functions $F$ of the form 
\begin{equation}\label{eq:Frec}
F(n):=\sum_{j=0}^{R-1} q_{j}(n) a_{n+j}.
\end{equation}
Such a function satisfies  $\vert F(n)\vert \le cRDH(n+1)^D\delta^{n+R}$ for all $n\ge 0$ and $F(n+p)\equiv F(n) \mod p$ for all prime number $p$ and all $n\ge 0$. Hence for all $N\ge 1$, $(F(0), \ldots, F(N-1))\in A(N,\underline{K})$ where 
$K_j:=cRDHj^D\delta^{j+R-1}$. Therefore, given $N$, if 
\begin{equation}\label{eq:prod1}
\prod_{j=1}^N\left(1+\frac{2K_j}{P_{j-1}}\right) < (2H+1)^{RD},
\end{equation}
there exists {\em two} different functions $F_1$ and $F_2$ of the form \eqref{eq:Frec} such that $F_1(n)=F_2(n)$ for all $n\in\{0,1, \ldots, N-1\}$. Hence the function $G_N:=F_1-F_2$ is of the form \eqref{eq:Frec} 
$$
G_N(n)=\sum_{j=0}^{R-1} q_{j}(n) a_{n+j},
$$ 
with $q_0(X), \ldots, q_{R-1}(X)\in \mathbb Z[X]$ not all identically zero, with $G_N(n)=0$ for every integer $n$ in $\{0,1, \ldots, N-1\}$ and
$\vert G_N(n)\vert \le 2cRDH(n+1)^D\delta^{n+R}$ for all $n\ge 0$. Note that $G_N$ depends on $N$ which is fixed but can be as large as desired in this construction.

Eq.~\eqref{eq:prod1} holds if we assume the stronger condition
\begin{equation}\label{eq:prod2}
\prod_{j=1}^\infty\left(1+\frac{2K_j}{P_{j-1}}\right) \le H^{RD},
\end{equation}
because the assumption $\delta<e$ implies the convergence of the product
$$
\Phi(D,x):=\prod_{j=1}^\infty\left(1+x\frac{j^D\delta^j}{P_{j-1}}\right) 
$$
for all $x\ge 0$ and all $D\ge 0$, and obviously $1+x\frac{j^D\delta^j}{P_{j-1}}\ge 1$. We shall provide an upper bound for $\Phi(D,x)$ in \S\ref{ssec:phidx}, from which we shall deduce values of $H,R$ and $D$ such that \eqref{eq:prod2} holds. It is important to observe here that \eqref{eq:prod2} does not depend on $N$.

\subsection{Proof that $G_N(n)=0$ for all $n\ge 0$}

Following the Perelli-Zannier method, we now want to prove that, provided 
$N$ is large enough, $G_N(n)$ vanishes for all $n\ge 0$. Assume this is not the case. Then for every $N$,  let $M_N\ge N$ denote the largest integer such that $G_N(0)=G_N(1)=\ldots=G_N(M_N)=0$ but $G_N(M_N+1)\neq 0$. We fix $\alpha$ in $\big(0, \frac{2}{\log(\delta)}-2\big)$. 
\medskip

We shall first prove that $G_N(m)=0$ for $m$ in $I:=[2M_N, (2+\alpha)M_N]$. Let $m\in I$. First assume that $p$ is a prime and $p<M_N$: since $G_N(0)=\cdots= G_N(p)=0$ and $G_N(n+p)\equiv G_N(n) \mod p$, $p$ divides $G_N(m)$. Assume now that $M_N\le p\le m/2$, so that 
$$
0\le m-2p\le m-2M_N\le \alpha M_N \le M_N,
$$
hence $G_N(m)\equiv G_N(m-2p)=0 \mod p$. Assume to finish that $m-M_N< p\le m$ (such primes have not yet been considered) so that $0\le m-p<M_N$ 
and $G_N(m)\equiv G_N(m-p)\equiv 0 \mod p$. It follows that $G_N(m)$ si divisible by $P_{m/2}P_m/P_{m-M_N}$.

Therefore, if $G_N(m)\neq 0$ for some $m\in I$, then 
$$
\vert G_N(m)\vert \ge e^{m/2+M_N+o(M_N)} \ge e^{2M_N+o(M_N)},
$$
where $o(M_N)$ denotes a term such that $o(M_N)/M_N$ becomes arbitrarily small when $N$ is taken arbitrarily large. But on the other hand, we know 
that, for any $m\in I$, 
$$
\vert G_N(m)\vert \le 2cRDH(m+1)^D\delta^{m+R}\le 2cRDH (2M_N+\alpha M_N+1)^D\delta^{R} \delta^{(2+\alpha) M_N}.
$$
We recall that we assume that $H,R$ and $D$ are such that \eqref{eq:prod2} holds, which is independent of~$N$. Hence we can let $N\to +\infty$, hence {\em a fortiori} $M_N\to +\infty$ so that the above lower and upper bounds for $G_N(m)\neq 0$ imply that $e^2\le \delta^{2+\alpha}$, \textit{i.e.} that 
$$
\alpha\ge \frac{2}{\log(\delta)}-2, 
$$
which is contrary to the assumption on $\alpha$. Hence, provided $N$ is large enough, we have $G_N(m)=0$ for all $m$ in $I$.

We thus have $G_N(m)=0$ for all integers $m$ in $[0, \ldots, M_N]$ or $[2M_N, (2+\alpha) M_N]$. It follows that for any $p\le (1+\alpha)M_N-1$, $p$ divides $G_N(M_N+1)$. Indeed, if $p \le M_N$, we  write $M_N+1=n+p$ 
for some $n\le M_N-2$ so that $G_N(M_N+1)\equiv G(n)=0\mod p$, while if 
$M_N<p\le (1+\alpha)M_N-1$, we have $0=G_N(M_N+1+p) \equiv G(M_N+1) \mod p$ because $M_N+1+p\in [2M_N, (2+\alpha) M_N]$. Hence, because $G_N(M_N+1)\neq 0$, we have 
$$
\vert G_N(M_N+1)\vert \ge P_{(1+\alpha) M_N-1} \ge e^{(1+\alpha)M_N+o(M_N)}.
$$
On the other hand, 
$$
\vert G_N(M_N+1) \vert \le  c2RDH (M_N+2)^D\delta^{M_N+1+R}. 
$$
As above, we take $N$ large enough so that these two bounds imply that $\alpha\le\log(\delta)-1$, which is impossible because $\log(\delta)-1<0$ while $\alpha$ was chosen positive.

Therefore, there is no such $M_N$ such that $G(M_N+1)\neq 0$, so that $G(n)=0$ for all integer $n\ge 0$.

\subsection{Upper bound for $\Phi(D,x)$}\label{ssec:phidx}

In this section, $x>0$ is a fixed real parameter and $D\ge 1$ is a fixed integer. We fix $\varepsilon>0$ such that $\delta<e-\varepsilon$, and we let $\omega=\delta/(e-\varepsilon)<1$. By the Prime Number  Theorem, we have $P_{j-1}\ge (e-\varepsilon)^{j}$ for all $j> J=J(\varepsilon)$ so that
$$
\Phi(D,x):= \prod_{j=1}^\infty \left(1+xj^D\frac{\delta^j}{P_{j-1}}\right)\le \prod_{j=1}^J \left(1+xj^D\frac{\delta^j}{P_{j-1}}\right) \prod_{j=J+1}^\infty \left(1+xj^D\omega^j\right).
$$
We have
\begin{align*}
\prod_{j=1}^J \left(1+xj^D\frac{\delta^j}{P_{j-1}}\right)\leq \prod_{j=1}^J \left(1+xj^D\delta^j\right)
\leq\prod_{j=1}^J \left(1+(1+x)j^D\delta^j\right)
\leq \prod_{j=1}^J \left(2(1+x)j^D\delta^j\right),
\end{align*}
since $(1+x)j^D\delta^j\geq 1$ for every $j$ in $\{1,\dots,J\}$. Hence, 
$$
\prod_{j=1}^J \left(2(1+x)j^D\delta^j\right)=2^J(1+x)^JJ!^D\delta^{J(J+1)/2}\leq 2^J(1+x)^JJ!^D\delta^{J^2}.
$$
In addition, we have
$$
\prod_{j=J+1}^\infty \left(1+xj^D\omega^j\right)\leq \prod_{j=1}^\infty \left(1+xj^D\omega^j\right),
$$
which yields 
$$
\Phi(D,x)\leq 2^J(1+x)^JJ!^D\delta^{J^2}\prod_{j=1}^{\infty} \left(1+xj^D\omega^j\right).
$$

We now bound the infinite product $\Psi(D,x):=\prod_{j=1}^{\infty} \left(1+xj^D\omega^j\right)$. The maximum of the function $t\mapsto t^D \omega^{t/2}$ is $m(D):=(2D/(e\log(1/\omega)))^D$, attained at $j_0:=2D/\log(1/\omega)$. Hence, for all $j\ge j_0$, we have $j^D\le m(D)\omega^{-j/2}$. Moreover, $t\mapsto t^D \omega^{t/2}$ is increasing on $[0,j_0]$ and $m(D)\le j_0^D$. Hence,  
\begin{align*}
\Psi(D,x) &\le \prod_{1\le j <j_0} \left(1+xj^D\omega^j\right) \prod_{j\ge j_0}^\infty \left(1+xm(D)\omega^{j/2}\right)
\\
&\le \left(1+(1+x) j_0^D\right)^{\lfloor j_0\rfloor} \prod_{j=1}^\infty 
\left(1+xm(D)\omega^{j/2}\right)
\\
&\le \left(2(1+x)j_0^D \right)^{\lfloor j_0\rfloor} \prod_{j=1}^\infty \left(1+xj_0^D\omega^{j/2}\right).
\end{align*}
We now bound $\prod_{j=1}^\infty\left(1+xj_0^D\omega^{j/2}\right)$. We set $y:=xj_0^D$. Since $t\mapsto \omega^t$ is decreasing on $[0,\infty)$, we have
\begin{align*}
\log\left(\prod_{j=1}^\infty \left(1+y\omega^{j/2}\right) \right)
&\le \int_0^{+\infty} \log(1+y\omega^{t/2}) dt 
\\
& \leq \frac{2}{\log(1/\omega)}\int_0^y \frac{\log(1+u)}u du,  \quad (u:=y \omega^{t/2})
\\
& \leq -\frac{2\textup{Li}_2(-y)}{\log(1/\omega)}.
\end{align*}
Here, we use the dilogarithm $\textup{Li}_2(z):=-\int_0^z \log(1-x)/x dx$ defined for $z\in \mathbb C\setminus (1, +\infty)$ using the principal 
branch of $\log$ in the integral;  see \cite[p. 1, (1.4)]{lewin}. Here, we want to use it  for large negative values $-y$. For this, we use the identity (see \cite[p. 4, (1.7)]{lewin})
$$
\textup{Li}_2(-y) =  - \frac{1}{2}\log(y)^2-\textup{Li}_2(-1/y) -\zeta(2), \quad y>0
$$
which yields
$$
-2\textup{Li}_2(-y) \leq \log(y)^2+4\zeta(2), \quad y\geq 1,
$$
because $ \textup{Li}_2(-1/y) +\zeta(2)  \le 2\zeta(2)$ when $y\ge 1$. We 
obtain that for $y\ge 1$, 
$$
\prod_{j=1}^\infty \left(1+y\omega^{j/2}\right) \le c_0 e^{\log(y)^2/\log(1/\omega)},
$$
with $c_0:=\exp(4\zeta(2)/\log(1/\omega))\ge 1$.

Putting all the pieces together, we finally obtain the bound
\begin{equation}\label{eq:majPhi}
1\le \Phi(D,x) \le 2^J(1+x)^J J!^D \delta^{J^2} (2(1+x)j_0^D)^{\lfloor j_0\rfloor}c_0e^{\log(xj_0^D)^2/\log(1/\omega)},
\end{equation}
where we recall that $\omega=\delta/(e-\varepsilon)$ where $\varepsilon 
> 0$ is such that $\delta< e-\varepsilon$.

\subsection{Conclusion of the proof} \label{ssec:conclusion1}

For ease of reading, we set $d,r,h$ for $D,R,H$. We want to find conditions on $d,r$ and $h$ such that $\Phi(d,x)\le h^{rd}$ when $x=2crdh\delta^{r-1}$ (which corresponds to \eqref{eq:prod2}). It will be enough to find conditions on $d,r,h$ and $\varepsilon$ such that the right-hand side of \eqref{eq:majPhi} is $\le h^{rd}$. 

From now on, we set $\ell:=\log(\delta)<1$. We assume that $d$ and $\rho$ depend on $\ell$ but are independent of $h$,  and we let $r:=\lfloor \rho\log(h)\rfloor+1$. Since $J$ and $j_0$ are also fixed, when $h\to +\infty$, 
we have
$$
\log \left(c_02^J(1+x)^J J!^d \delta^{J^2} (2(1+x)j_0^d)^{\lfloor j_0\rfloor}e^{\log(xj_0^d)^2/\log(1/\omega)}\right) \sim \frac{(1+\rho\ell)^2}{\log(1/\omega)} \log(h)^2,
$$
while 
$$
\log(h^{rd}) \sim d\rho\log(h)^2.
$$
Hence for our goal, it suffices to choose $d,\rho$ and $\varepsilon$ such 
that
\begin{equation}\label{eq:rho1}
\frac{(1+\rho\ell)^2}{\log(1/\omega)} < d\rho.
\end{equation}

Recall that $\omega=\delta/(e-\varepsilon)$ so that $\log(1/\omega)\to 1-\ell$ as $ \varepsilon\to 0$. So, by choosing $d$ and $\rho$ such that
\begin{equation}\label{eq:rho2}
\frac{(1+\rho\ell)^2}{1-\ell}\leq d\rho,
\end{equation}
we can choose $\varepsilon>0$ such that Eq.~\eqref{eq:rho1} holds true. Eq.~\eqref{eq:rho2} is equivalent to
$$
\ell^2\rho^2+(2\ell-d(1-\ell))\rho+1 \leq 0,
$$
which defines a polynomial in $\rho$ whose discriminant is $\Delta:=d(1-\ell)(d(1-\ell)-4\ell)$. Taking $d:=\max(1,\lceil\frac{4\ell}{1-\ell}\rceil)$ ensures that \eqref{eq:rho2} holds true for any choice of $\rho>0$ in
$$
\left[\frac{d(1-\ell)-2\ell-\sqrt{\Delta}}{2\ell^2},\frac{d(1-\ell)-2\ell+\sqrt{\Delta}}{2\ell^2} \right].
$$
For simplicity, we also restrict $\rho$ to be $\le 1/\ell$ which is possible because the product of those roots is $1/\ell^2$ and, since $d(1-\ell)\geq 4\ell$, we have 
\begin{equation*}
\frac{d(1-\ell)-2\ell+\sqrt{\Delta}}{2\ell^2}\geq\frac{1}{\ell}\quad\textup{and}\quad\frac{d(1-\ell)-2\ell-\sqrt{\Delta}}{2\ell^2}\leq\frac{1}{\ell}.
\end{equation*}
With such choices of $\varepsilon$, $d$ and $\rho$, we now define $H(c,\delta)$ as the smallest integer $h\ge 1$ such that
\begin{equation}\label{eq:defHdelta}
c_02^J(1+x)^J J!^d \delta^{J^2} (2(1+x)j_0^d)^{\lfloor j_0\rfloor}e^{\log(xj_0^d)^2/\log(1/\omega)}  \le h^{rd},
\end{equation}
where $x=2crdh\delta^{r-1}$. 
We then obtain \eqref{eq:upperbounds} with $\max \deg(p_j)\le d-1=\max(0,\lceil\frac{5\log(\delta)-1}{1-\log(\delta)}\rceil)$ and $S\le r-1\le \lfloor \rho\log(h)\rfloor\le  \log(H(c,\delta))/\ell$. Notice that $H(c,\delta)$ also depends on the choice of $\rho$ and we now explain how to bound it.
\medskip

The left-hand side of \eqref{eq:defHdelta} is an increasing function of $h\ge 1$, which appears in the expressions of $r:=\lfloor \rho \log(h)\rfloor+1$ and $x:=2crdh\delta^{r-1}$. Hence, 
$$
H(c,\delta)^{d(1+\lfloor \rho\log(H(c,\delta))\rfloor)} 
$$
is larger than the value $A$ (which is $\ge 1$) of the left-hand side of \eqref{eq:defHdelta} at $h=1$, in which case 
$r=1$ and $x=2cd$. It follows that 
$\log(A)\le d\log(H(c,\delta))+d\rho\log(H(c,\delta))^2$, so that  
\begin{equation}\label{eq:boundlowH}
H(c,\delta)\ge \exp\Big(\frac{\sqrt{d^2+4d\rho\log(A)}-d}{d\rho}\Big).
\end{equation}
Since $A\to +\infty$ when $\delta\to e$ and $\varepsilon\to 0$ (because of the term $1/\log(\omega)$), $H(c,\delta)$ can be very large.

We now explain how to bound $Y:=H(c,\delta)$ from above. We assume that $H(c,\delta)\ge 2$ otherwise there is nothing else to do.
The left-hand side of \eqref{eq:defHdelta} is greater than or equal to~$1$, so that we can take the logarithms of both sides. After some transformations, we obtain a function $S(h)\ge 0$ for all $h\ge 1$ (which could be explicited) such that $Y$ is the smallest integer $h\ge 1$ such that 
\begin{equation}\label{eq:1012}
S(h)\le \left(\rho d-\frac{(1+\rho\ell)^2}{\log(1/\omega)}\right)\log(h)^2.
\end{equation}
Recall that 
$$
\gamma:=\rho d-\frac{(1+\rho\ell)^2}{\log(1/\omega)}>0.
$$ 
Moreover, there exist $\alpha$ and $\beta$ that depend on $\rho, d, \delta$ and $\varepsilon$ (and could be explicited as well) such that $S(h)\le \alpha \log(h)+\beta$ for all $h\ge 1$. Since $Y\ge 2$ is the smallest integer such that \eqref{eq:1012} holds, we have 
$$
\gamma\log(Y-1)^2<S(Y-1)\le \alpha \log(Y-1)+\beta.
$$
Hence, $\log(Y-1)$ is smaller than the largest solution of the quadratic equation 
$$
\gamma X^2- \alpha X-\beta=0, 
$$
so that finally
\begin{equation}\label{eq:boundupH}
H(c,\delta) \le 1+\exp\left( \frac{\alpha+\sqrt{\alpha^2+4\beta\gamma}}{2\gamma}\right).
\end{equation}

\section{Proof of Theorem \ref{theo:3}} \label{sec:prooftheo3}

Let $F(z)$ be as in the theorem such that $F(n)=a_n$ for all $n>u$. Notice that $\widetilde{a}_n:=a_{n+\lceil  u \rceil+2}$, $n\ge0$, is a primary peudo-polynomial; it is a polynomial if and only if $(a_n)_{n\geq 0}$ is a polynomial. The function $\widetilde{F}(z):=F(z+\lceil  u \rceil+2)$ is analytic in $\Re(z)>-2$,  satisfies $\vert \widetilde{F}(z)
\vert \le \tilde{c}\cdot \exp(\rho \Re(z))$ in $\Re(z)>-2$  for some constant $\tilde{c}>0$, and $\widetilde{a}_n=\widetilde{F}(n)$ for all $n\ge 0$. Moreover, $F(z)$ is  in $\mathbb Q[z]$ if and only if $\widetilde{F}(z)$ is in $\mathbb Q[z]$. Therefore, without loss of generality, we can 
and will assume that $F(z)$ is analytic in $\Re(z)>-2$ and that $F(n)=a_n$ for all integers $n\ge 0$.

Let $\mathcal{C}_n$ denote the circle of center $n$ and radius $n$ oriented in the direct sense. The function $F(z-1)$ being analytic in $\Re(z)>-1$, the residue theorem yields
$$
\frac{(n-1)!}{2i\pi}\int_{\mathcal{C}_n} \frac{F(z-1)}{(z-1)\cdots (z-n)}dz=\sum_{k=0}^{n-1}(-1)^{n-k-1}\binom{n-1}{k} a_k:=b_{n-1}\in \mathbb Z. 
$$
We parametrize the circle $\mathcal{C}_n$ as $n+ne^{2ix}=2n\cos(x)e^{ix}$ for $x\in [-\pi/2, \pi/2]$. N\"orlund \cite[p.~387]{norlund} proved that 
$$
\left \vert \frac{(n-1)!}{(z-1)\cdots (z-n)} \right\vert \le  c_1(n) e^{-2n\cos(x)\psi(x)}, \quad  z=2n\cos(x)e^{ix}\in \mathcal{C}_n,
$$
where $c_1(n)>0$ is bounded above by some polynomial in $n$ and 
$$
\psi(x):=\cos(x)\log(2\cos(x))+x\sin(x).
$$
(See also the proof given in \cite{rivoalwelter}.)
The minimum on $[-\pi/2,\pi/2]$ of $\psi(x)$ is $\log(2)$ at $x=0$. 
 
We have 
$$
\vert F(z-1)\vert\le c\cdot e^{\rho\Re(z-1)} =ce^{-\rho}\cdot e^{2n\rho\cos(x)^2}, \quad  z=2n\cos(x)e^{ix}\in \mathcal{C}_n.
$$ 
Hence
$$
\vert b_{n-1}\vert \le c_2(n) \max_{x\in [-\pi/2,\pi/2]} 
e^{2n \cos(x)(\rho\cos(x)-\psi(x))}
$$ 
where $c_2(n)>0$ is bounded above by some polynomial in $n$. 
Notice that $0\le \cos(x)\le 1$ on $[-\pi/2,\pi/2]$. Recall also that $\rho> 0$. Hence if $x$ is such that $\rho\cos(x)-\psi(x)\ge 0$, then 
$$
\cos(x)(\rho\cos(x)-\psi(x))\le \rho-\psi(x)\le \rho-\log(2), 
$$
whereas if $x$ is such that $\rho\cos(x)-\psi(x)\le 0$, then $\cos(x)(\rho\cos(x)-\psi(x))\le 0$. Therefore, for all $x\in [-\pi/2,\pi/2]$, 
$$ 
e^{2\cos(x)(\rho\cos(x)-\psi(x))} \le \max(1,e^{2(\rho-\log(2))})
$$
and
$
\vert b_{n-1}\vert \le c_2(n)  \max(1,e^{2(\rho-\log(2))})^n.
$
Since $2(\rho-\log(2)) <1$,  it follows that
$$
\limsup_{n\to +\infty} \vert b_n\vert^{1/n}<e.
$$
But because $(a_n)_{n\ge0}$ is a primary pseudo-polynomial, we know by Theorem \ref{theo:1}$(ii)$ that if $b_n$ is not eventually equal to $0$ then 
$$
\liminf_{n\to +\infty} \vert b_n\vert^{1/n} \ge e.
$$
This implies that $b_n$ is indeed eventually equal to $0$, thus that there exist $P(X) \in \mathbb Q[X]$ and an integer $N\ge 0$ such that $a_n=P(n)$ for all $n\ge 0$. 

Consider now the function $g(z):=F(z)-P(z)$ which is analytic in $\Re(z)>-2$ (at least), and such that $g(n)=0$ for every integer $n\ge 0$. Moreover, since $\rho>0$, there exists a constant $d>0$ such that $\vert g(z)\vert \le d\cdot \exp(\rho \vert z\vert)$ for any $z$ such that $\Re(z)>0$. Since $\rho<\frac{1}{2}+\log(2)<\pi$, we can then apply a classical result of Carlson (see Hardy \cite[p.~328]{hardy}) and deduce that $g(z)=0$ identically. Hence $F(z)$ reduces to a polynomial function in $\mathbb Q[z]$. This completes the proof of Theorem~\ref{theo:3}.

\section{Construction of genuine primary pseu\-do-po\-ly\-nomials} 
\label{sec:prooftheo51}

We conclude this paper by presenting a method to construct a non-polynomial primary pseudo-po\-ly\-nomial starting from a given primary pseudo-polynomial $(a_n)_{n\ge 0}$ such that $a_0=1$. The justification of the method uses a non-trivial property satisfied by $E$-functions.

Let as usual $(b_n)_{n\ge 0}$ be the binomial transform \eqref{eq:binomialtransform} of $(a_n)_{n\ge 0}$. Let 
$$
F_b(x):=\sum_{n=0}^\infty \frac{b_n}{n!}x^n. 
$$
We assume that $a_0=1$, so that $b_0=1$ as well.  We  define the sequence $(c_n)_{n\ge 0}$ formally by 
$$
F_c(x):=\sum_{n=0}^\infty \frac{c_n}{n!}x^n=\frac{1}{F_b(x)}. 
$$
Let us now define $(u_n)_{n\ge 0}$ as the inverse binomial transform \eqref{eq:involution}  of $(c_n)_{n\ge 0}$, \textit{i.e.} 
$$
u_n:=\sum_{k=0}^n \binom{n}{k}c_k.
$$ 
Then we have the following.

\begin{theo} \label{theo:51} Let $(a_n)_{n\geq 0}\in\mathbb{Z}^{\mathbb{N}}$ be a primary pseudo-polynomial such that $a_0=1$. Then, $(u_n)_{n\ge 0}\in \mathbb Z^{\mathbb{N}}$ is a primary pseudo-polynomial. Moreover, 
assuming also that $f_a(x)$ is a $G$-function, if $\limsup_n \vert u_n\vert^{1/n}<e$, then $a_n=u_n=1$ for all $n\ge 0$.
\end{theo} 
Consequently, if $f_a(x)$ is a $G$-function not equal to $(1-x)^{-1}$, 
then $(u_n)_{n\ge 0}\in \mathbb Z^{\mathbb{N}}$ is a primary pseudo-polynomial such that $\limsup_n \vert u_n\vert^{1/n}\ge e$ (hence not a polynomial). We explain after the proof of the theorem why $u_n$ grows like $n!$ in this case.

So far, the assumption that $f_a(x)$ is a $G$-function is known to be satisfied only when $\limsup_n \vert a_n\vert^{1/n}<e$ (when the Perelli-Zannier Theorem can be applied), which in turn implies that $(a_n)_{n\geq 0}$ should be a polynomial by Conjecture~\ref{conj:ruzsaext}. Hence, in 
practice the second assertion of Theorem~\ref{theo:51} is useful only when $(a_n)_{n\geq 0}$ is already known to be a polynomial in which case 
$F_b(x)$ $\in \mathbb Q[x]$. For instance, if $a_n=n+1$, then 
\begin{equation}\label{eq:1312}
\sum_{n=0}^\infty \frac{c_n}{n!}x^n=\frac{1}{1+x}\quad\textup{and}\quad u_n=\sum_{k=0}^n (-1)^{k}\binom{n}{k}k!.
\end{equation}

\begin{proof}[Proof of Theorem \ref{theo:51}]

We first prove that for all $n\ge 0$, $c_n\in \mathbb Z$ and for every prime $p\le n$, $p$ divides $c_n$. 
By definition of the $c_n$'s, we have
$$
1=\left(\sum_{n=0}^\infty\frac{b_n}{n!}x^n\right)\left(\sum_{n=0}^\infty\frac{c_n}{n!}x^n\right)=\sum_{n=0}^\infty\frac{1}{n!}\left(\sum_{k=0}^n\binom{n}{k}b_kc_{n-k}\right)x^n,
$$
so that $c_0=1$ and, for every integer $n\geq 1$, we have
$$
\sum_{k=0}^n\binom{n}{k}b_kc_{n-k}=0.
$$
This yields the recursive relation (because $b_0=1$)
$$
c_n=-\sum_{k=1}^n\binom{n}{k}b_kc_{n-k},\quad(n\geq 1).
$$
It follows that, for all $n\geq 0$, $c_n$ is an integer. 

Let $n$ be a positive integer such that, for every positive integer $m< n$ and every prime number $p\leq m$, $p$ divides $c_m$. Consider a prime number $p\leq n$ and an integer $k$ in $\{1,\dots,n\}$. If $p\leq k$ then $p$ divides $b_k$. If $p\leq n-k$, then $p$ divides $c_{n-k}$. If $p>\max(k,n-k)$, then $p$ divides $\binom{n}{k}$ because $p$ divides $n!$ but neither $k!$ nor $(n-k)!$. In all cases, $p$ divides $\binom{n}{k}b_kc_{n-k}$, so that $p$ divides $c_n$. By strong induction on $n$, it follows that, for every integer $n\geq 0$ and every prime $p\leq n$, $p$ divides $c_n$ (this property holds trivially if $n=0$ or $1$). By Theorem~\ref{theo:1}$(i)$, the sequence $(u_n)_{n\geq 0}\in\mathbb{Z}^{\mathbb N}$ is a primary pseudo-polynomial. Notice that so far we had no need to assume that $f_a$ 
is a $G$-function.

We  now complete the proof of Theorem \ref{theo:51} when $f_a$ is also a $G$-function. We first observe that $F_b(x)$ is an $E$-function: indeed, $f_b(x)$ is a $G$-function because $$
f_b(x):=\sum_{n=0}^\infty b_n x^n=\frac{1}{1+x}f_a\left(\frac{x}{1+x}\right)
$$ 
and $f_a(x)$ is a $G$-function.~(\footnote{Given a $G$-function $f(x)$ and $\alpha(x)$ an algebraic function over $\Qbar(x)$ regular at $x=0$,  $f(x\alpha(x))$ is a $G$-function.})
Let us assume that $\limsup_n \vert u_n\vert^{1/n}<e$. 
By the Perelli-Zannier Theorem quoted in the Introduction, $f_u(x):=\sum_{n=0}^\infty u_n x^n$ is a $G$-function. From the equation $c_n:=\sum_{k=0}^n (-1)^{n-k} \binom{n}{k}u_k$ ($\forall n\ge0)$, we deduce that 
$$
f_c(x):=\sum_{n=0}^\infty c_n x^n=\frac{1}{1+x}f_u\left(\frac{x}{1+x}\right)
$$ 
is also a $G$-function. 
Hence $F_c(x):=1/F_b(z)$ is also an $E$-function. 
Therefore, $F_b(x)$ is a unit of the ring of $E$-functions, \textit{i.e.} 
it is of the form $\beta e^{\alpha z}$ with $\alpha\in \Qbar$ and $\beta\in \Qbar^*$ by \cite[p. 717]{andre} (see also Footnote 4). Therefore, $b_n= \beta \alpha^n$ for all $n\ge 0$, with the usual convention that $\alpha^0=1$ if $\alpha=0$. Since all primes $p\le n$ divide $b_n$, we deduce that $\alpha=0$ is the only possibility. Hence 
$1=b_0=\beta$ and $b_n=0$ for all $n\ge 0$, so that $a_n=1$ for all $n\ge 0$. It now remains to observe the following facts: 
``$a_n=1$ for all $n\ge 0$'' is equivalent to ``$b_0=1$ and $b_n=0$ 
for all $n\ge 1$'' which is equivalent to ``$c_0=1$ and $c_n=0$ for all $n\ge 1$'', which in turn is equivalent to ``$u_n=1$ for all $n\ge 0$''.
\end{proof}

To conclude, let us explain how to obtain the asymptotic behavior of $u_n$ as $n\to +\infty$. 
Since 
$$
F_u(x):=\sum_{n=0}^\infty\frac{u_n}{n!}x^n=e^{x}F_c(x)=e^{x}/F_b(x),
$$ 
the asymptotic behavior of $u_n/n!$ is determined by the zeroes of smallest modulus of $F_b(x)$, when it has at least one. Notice that 
an $E$-function $F(x)$ with no zero in $\mathbb C$ must be of the form $\beta e^{\alpha x}$ with $\alpha\in \Qbar$ and $\beta\in \Qbar^*$. Indeed, an $E$-function is an entire function satisfying $\vert F(x)\vert \ll e^{\rho\vert x\vert}$ for some $\rho>0$, so that Hadamard's factorization theorem yields $F(x)=\beta x^m e^{\alpha x}\prod_{j\ge 1}\big((1-\frac{x}{x_j})e^{x/x_j}\big)$ where the $x_j$'s are the zeroes of $F(x)$, $m\in \mathbb N$ and $\alpha, \beta\in \mathbb C$.~(\footnote{This argument also explains the characterization of the units of the ring of $E$-functions: we simply have to observe that if an $E$-function $F(x)$ is a unit, \textit{i.e.} that $1/F(x)$ is an $E$-function, then it does not vanish anywhere on $\mathbb {C}$ because an $E$-function is an entire function.}) Therefore, a ``no zero'' assumption implies that $F(x)=\beta e^{\alpha x}$ with $\beta \neq 0$, and since $F(x)\in \Qbar[[x]]$, $\alpha, \beta$ must be 
in $\Qbar$. Coming back to $F_b(x)$, we have seen during the proof of Theorem \ref{theo:51} that this case implies that $\alpha=0,\beta=1$, hence that $f_a(x)=(1-x)^{-1}.$

Therefore, assuming that $f_a(x)$ is a $G$-function different of $(1-x)^{-1}$, the $E$-function $F_b(x)$ has at least one zero. Let $x_1, \ldots, x_m$ denote the zeroes of $F_b(x)$ of the same modulus which is the smallest amongst all modulus of the zeroes. Classical transfer theorems in \cite[Chapter VI]{bible} enable to deduce the asymptotic behavior of $u_n$. For instance, if the $x_j$'s are simple zeroes of $F_b(x)$, then 
$$
u_n=n!\sum_{j=1}^m \frac{e^{x_j}}{F_b'(x_j)}\frac{1+o(1)}{x_j^{n}}.
$$
This is coherent with \eqref{eq:1312} above (where $F_b(x)=1+x$)  because we can rewrite it as 
$$
u_n =  (-1)^n n!\sum_{k=0}^n\frac{(-1)^{n-k}}{(n-k)!}\sim \frac{(-1)^n}{e} n!, \quad n\to +\infty.
$$
Because of the different arguments of the $x_j$'s, oscillations can occur. In presence of zeroes of $F_b(x)$ of higher multiplicities, similar but more complicated  expressions can be given. Finally, even though $F_b(x)/e^{x}$ is an $E$-function, we don't expect $e^x/F_b(x)$ to be $D$-finite in general,~(\footnote{A classical result of Harris-Sibuya~\cite{harrissibuya} ensures that if $y$ and $1/y$ are both holonomic, then $y'/y$ is 
an algebraic function.}) but this is obviously the case if $F_b(x)$ is a polynomial.

\medskip

\noindent {\bf Acknowledgements.} Both authors have partially been funded 
by the ANR project {\em De Rerum Natura} (ANR-19-CE40-0018). This project 
has received funding from the European Research Council (ERC) under the European Union’s Horizon 2020 research and innovation programme under the Grant Agreement No 648132.

\medskip

\noindent E. Delaygue, Institut Camille Jordan, 
Universit\'e Claude Bernard Lyon 1, 
43 boulevard du 11 novembre 1918,
69622 Villeurbanne cedex, France

\medskip

\noindent T. Rivoal, Institut Fourier, CNRS et Universit\'e Grenoble Alpes, CS 40700, 38058 Grenoble cedex 9, France.

\medskip

\noindent Keywords: Primary pseudo-polynomial, Algebraic series, $G$-functions, $E$-functions,  Newton Interpolation.

\medskip

\noindent MSC 2020:  11A41, 11B50 (Primary); 11B37, 33E20, 41A05 (Secondary).

\end{document}